\documentclass[twoside,10pt]{amsart}

\usepackage{amsthm,amsmath,amsfonts,epsfig,amssymb}
\usepackage[T1]{fontenc}
\usepackage{graphicx}
\usepackage{enumerate}
\usepackage{xy}
\usepackage{graphics}
\xyoption{all}
\entrymodifiers={+!!<0pt,\fontdimen22\textfont2>}

\newtheorem{thm}{Theorem}[section]

\newtheorem{prop} [thm]{Proposition}
\newtheorem{lem} [thm]{Lemma}
\newtheorem{cor}[thm]{Corollary}

\newtheorem*{thm*}{Theorem}
\theoremstyle{definition}

\theoremstyle{remark}
\newtheorem{rem}[thm]{Remark}
\newtheorem{exem}[thm]{Example}

\def\C{\mathbb{C}}
\def\R{\mathbb{R}}

\def\Z{\mathbb{Z}}
\def\N{\mathbb{N}}

\def\P{\mathbb{P}}
\def\A{\mathbb {A}}
\def\O{\mathcal O}

\def\homm#1#2#3{\mathrm{Hom}_{#1}(#2,#3)}

\def\exten#1#2#3#4{\mathrm{Ext}_{#2}^{#1}(#3,#4)}

\def\spec#1{\mathrm{Spec}(#1)}
\def\carac#1{\mathrm{char}(#1)}

\begin{document}

\title{Mennicke symbols, $K$-cohomology and a Bass-Kubota theorem}
\author{J. Fasel}
\date{}

\address{Jean Fasel\\
Mathematisches Institut der Universit\"at M\"unchen\\ 
Theresienstrasse 39\\ 
D-80333 M\"unchen}
\email{jean.fasel@gmail.com}

\begin{abstract}
If $A$ is a smooth algebra of dimension $d\geq 2$ over a perfect field $k$ of characteristic different from $2$, then we show that the universal Mennicke symbol $MS_{d+1}(A)$ is isomorphic to the $K$-cohomology group $H^d(A,K_{d+1})$. When $k$ is algebraically closed and $S$ is a smooth surface, we further prove an analogue of the Bass-Kubota theorem for curves. Namely, there is a natural isomorphism $MS_2(A)\simeq Um_2(S)/SL_2(S)\cap E_3(S)$.  
\end{abstract}

\maketitle

\pagenumbering{arabic}



\section*{Introduction}

As its title indicates, this paper mainly deals with Mennicke symbols. These symbols were extensively studied, for instance in the solution of the congruence subgroup problem (\cite{Bass67}), the relations between Milnor $K$-theory and Quillen $K$-theory (\cite{Suslin82b}) or the computation of orbit sets of unimodular rows (\cite{Suslin80}). One of the goals of this article is to show that if $A$ is a smooth algebra of dimension $d\geq 4$ over a perfect field $k$ of characteristic different from $2$, then the universal Mennicke symbol $MS_{d+1}(A)$ has a cohomological interpretation. More precisely, it is isomorphic to the $K$-cohomology group $H^d(A,K_{d+1})$ (here it doesn't matter if we consider Milnor or Quillen $K$-groups). In this sense, this paper is a sequel of \cite{Fasel08b} where it was shown that the universal weak Mennicke symbol has a cohomological interpretation in the situation above. Weak Mennicke symbols were introduced by van der Kallen in order to understand orbit sets of unimodular rows, and it is clear by construction that there is a surjective homomorphism from the universal weak Mennicke symbol to the universal Mennicke symbol. Our cohomological approach allows us to compare these two symbols and derive some results on prestabilization in $K_1$. We are also able to compute the universal Mennicke symbol in some situations, for instance in the case of smooth rational real algebras, as well as to show that this symbol can be big over fields of small cohomological dimension.  The methods we use can be seen as an extension of the methods used in \cite{Fasel10b} to prove that stably free modules of rank $d-1$ are free over smooth affine algebras of dimension $d$ over an algebraically closed field. 

In the second part of the paper, we also study the Mennicke symbols of length $2$. We use Grothendieck-Witt groups to prove that if $S$ is a smooth surface over an algebraically closed field $k$ then $SL_2(S)\cap ESp_4(S)=SL_2(S)\cap E_3(S)$. This allows to deduce that 
$$MS_2(S)=Um_2(S)/SL_2(S)\cap E_3(S)=Um_2(S)/SL_2(S)\cap ESp_4(S),$$ 
which is the analogue of the Bass-Kubota theorem for Dedekind rings (see \cite{Kubota65} for instance). We expect also that the Bass-Kubota theorem will be true relative to an ideal, thus allowing to extend the results of \cite{Suslin79} about reciprocity laws to smooth surfaces. Nevertheless, the proof of this fact will require a Bass-Kubota theorem for some singular surfaces and we don't know at the moment how to deal with this problem. We hope to be able to do it in further work. 

The organization of the paper is as follows. In Section \ref{symbols}, we first recall the definitions of a Mennicke symbol and of a weak Mennicke symbol of length $n\geq 3$ associated to a unimodular row of length $n\geq 3$. We then recall the definitions of the sheaves we will need for the comparison theorem. In particular, we introduce the $K$-cohomology groups $H^i(X,K_j)$ associated to a regular scheme $X$. We then state and prove the comparison theorem, which says that if $A$ is a smooth algebra of dimension $d\geq 3$ over a perfect field $k$ with $\carac k\neq 2$, then the universal Mennicke symbol of length $d+1$ is isomorphic to $H^d(A,K_{d+1})$. The proof is an easy consequence of the description of the universal weak Mennicke symbol obtained in \cite[Theorem 4.9]{Fasel08b}.

In Section \ref{stab}, we use cohomological methods to compare the universal weak Mennicke symbol of length $d+1$ with the universal Mennicke symbol of length $d+1$ over a smooth algebra of dimension $d$. We also obtain divisibility results for universal Mennicke symbols. We then perform the computation of $MS_{d+1}(A)$ when $A$ is a smooth rational real algebra of dimension $d$ before focusing on stabilization questions. We finally prove that over smooth algebras of odd dimension $d$ over a field, $MS_{d+1}(A)$ is sufficient to understand the stably free modules of rank $d$. 

The proof of the Bass-Kubota theorem for smooth surfaces over an algebraically closed field takes place in Section \ref{BKS}. We start with a quick reminder of basic results on higher Grothendieck-Witt groups, including the definitions. These groups are used to prove that if $S$ is a smooth surface over an algebraically closed field of characteristic different from $2$, then $K_1Sp(S)$ injects into $SK_1(S)$. We then introduce Mennicke symbols of length $2$ before proving the Bass-Kubota theorem.  

\subsection*{Conventions}
The fields considered are of characteristic different from $2$. If $X$ is a scheme and $x_p\in
X^{(p)}$, we denote by $\mathfrak m_p$ the maximal ideal in
$\O_{X,x_p}$ and by $k(x_p)$ its residue field. Finally $\omega_{x_p}$
will denote the $k(x_p)$-vector space $\wedge^p(\mathfrak
m_p/\mathfrak m_p^2)$ (which is one-dimensional if $X$ is regular at
$x_p$).

\subsection*{Acknowledgments}
It is a pleasure to thank Ravi Rao for useful conversations, and also for motivating me to write the results of this paper. Part of this work was achieved during a beautiful tour at Penn State University and Chicago University. Special thanks are due P. Baum, L. Vaserstein, M. P. Murthy and M. V. Nori for making this possible. This work was supported by the Swiss National Science Foundation, grant PAOOP2\_129089.


\section{Symbols and cohomology}\label{symbols}

In this section, we recall the definitions of the Mennicke symbols and the weak Mennicke symbols. We then explore their links to the cohomology groups of some sheaves. The main result is the comparison theorem, which says that for a smooth affine algebra $A$ of dimension $d$ over some field $k$, the universal Mennicke symbol of length $d$ is isomorphic to the $K$-cohomology group $H^d(A,K_{d+1})$.

\subsection{Unimodular rows}\label{umn}
Let $R$ be a ring. A unimodular row of length $n\geq 2$ is a row $a=(a_1,\ldots,a_n)$ with $a_i\in R$ such that there exist $b_1,\ldots,b_n\in R$ such that $\sum a_ib_i=1$. We denote by $Um_n(R)$ the set of unimodular rows of length $n$ and we consider it as a pointed set with base point $e_1:=(1,0,\ldots,0)$. If $M\in GL_n(R)$ and $a\in Um_n(R)$, then $aM$ is also unimodular and thus $Um_n(R)$ is endowed with an action of $GL_n(R)$. Of course, any subgroup of $GL_n(R)$ acts, and one is classically interested in the (pointed) orbit set $Um_n(R)/SL_n(R)$, which classifies up to isomorphism projective modules such that $P\oplus R\simeq R^n$, and $Um_n(R)/E_n(R)$ which classifies unimodular rows up to elementary homotopies and is endowed with an abelian group structure when $n$ is reasonable compared to the Krull dimension $d$ of $R$ (\cite[Theorem 4.1]{vdKallen89}). A way to understand unimodular rows is through symbols, as explained in the next two sections.

\subsection{Mennicke symbols}\label{ms}

A Mennicke symbol of length $n\geq 3$ is a pair $(\phi,G)$ where $G$ is a group and  
 $$\phi:Um_n(R)\to G,$$
is a map such that the two following properties are satisfied:
\begin{enumerate}[ms1.]
\item $\phi(1,0,\ldots,0)=1$ and $\phi(v)=\phi(wM)$ if $M\in E_n(R)$. 
\item $\phi(a,a_2,\ldots,a_n)\cdot \phi(b,a_2,\ldots,a_n)=\phi(ab,a_2,\ldots,a_n)$ for any unimodular rows $(a,a_2,\ldots,a_n)$ and $(b,a_2,\ldots,a_n)$.
\end{enumerate}

It is clear from the definition that a universal Mennicke symbol $(ms,MS_n(R))$ exists. 

\subsection{Weak Mennicke symbols}

A weak Mennicke symbol of length $n\geq 3$ is a pair $(\phi,G)$ where $G$ is a group and
 $$\phi:Um_n(R)\to G,$$
is a map such that the two following properties are satisfied:
\begin{enumerate}[wms1.]
\item $\phi(1,0,\ldots,0)=1$ and $\phi(v)=\phi(wM)$ if $M\in E_n(R)$. 
\item $\phi(a,a_2,\ldots,a_n)\cdot \phi(1-a,a_2,\ldots,a_n)=\phi(a(1-a),a_2,\ldots,a_n)$ for any unimodular row $(a,a_2,\ldots,a_n)$ such that $(1-a,a_2,\ldots,a_n)$ is also unimodular.
\end{enumerate}

The definition of a weak Mennicke symbol, due to W. van der Kallen (\cite[\S1.3]{vdKallen89}), was originally more complicated. However, van der Kallen could show that in most of the situations the original definition is the same as the definition given above (\cite[Theorem 3.3]{vdKallen02}).

Again, it is clear that a universal weak Mennicke symbol $(wms,WMS_n(R))$ exists. Since a Mennicke symbol of length $n$ is obviously a weak Mennicke symbol of length $n$, there is a unique homomorphism $F:WMS_n(R)\to MS_n(R)$ such that the diagram
$$\xymatrix{Um_n(R)\ar[r]^-{wms}\ar[rd]_-{ms} & WMS_n(R)\ar[d]^-f \\ & MS_n(R)}$$
commutes. Observe that $f$ is surjective by definition. 

\subsection{Cohomology}\label{cohomology}

Let $F$ be a field. For any $n\in \N$, we denote by $K_n(F)$ the $n$-th Milnor $K$-theory group of $F$. We also denote by $W(F)$ the Witt ring of $F$, and for any $n\in \Z$ by $I^n(F)$ the $n$-th power of the fundamental ideal in $W(F)$ (with the convention that $I^n(F)=W(F)$ when $n\leq 0$). If $L$ is a $F$-vector space of dimension $1$, then we can also consider the twisted Witt group $W(F,L)$ (classifying vector spaces $V$ endowed with symmetric anisotropic isomorphisms $V\to \homm FVL$). Choosing a generator of $L$ yields an isomorphism $W(F)\to W(F,L)$ under which we can consider the images of $I^n(F)$. The resulting subgroups $I^n(F,L)\subset W(F,L)$ are independent of the choice of the generator of $L$ (\cite[Lemma E.1.2]{Fasel08a}).

By definition, $I^{n+1}(F,L)\subset I^n(F,L)$ for any $n$ and it turns out that the quotient is canonically isomorphic to $I^n(F)/I^{n+1}(F):=\overline I^n(F)$ (\cite[Lemma E.1.3]{Fasel08a}). There is a homomorphism $s_n:K_n(F)\to \overline I^n(F)$ defined on symbols by (\cite[Theorem 4.1]{Milnor69})
$$s_n(\{a_1,\ldots,a_n\})=\langle -1,a_1\rangle\otimes \ldots\otimes\langle -1,a_n\rangle.$$ 
Note that $s_n$ induces an isomorphism $K_n(F)/2K_n(F)\to \overline I^n(F)$ by \cite{Voevodsky03b} and \cite{Orlov07}. We define $G^n(F,L)$ following \cite[\S1]{Barge00} as the fibre product
$$\xymatrix{G^n(F,L)\ar[r]\ar[d] & I^n(F,L)\ar[d] \\
K_n(F)\ar[r]_-{s_n} & \overline I^n(F)}$$
Observe that $G^n(F,L)$ coincides with the (twisted) Milnor-Witt $K$-group $K_n^{MW}(F,L)$ defined in \cite[Definition 5.1]{Morel04} (see \cite[Theorem 5.3]{Morel04}).

Let $X$ be a regular connected scheme over $k$. If $U\subset X$ is an open subset we can consider the kernel of the residue homomorphism (\cite[Lemma 2.1]{Milnor69})
$$\xymatrix{K_n(k(U))\ar[r]^-{d_K} & \bigoplus_{x\in U^{(1)}} K_{n-1}(k(x))}.$$
This defines a presheaf on $X$, whose associated sheaf we denote by $K_n$. For the properties of the cohomology groups of this sheaf, we refer the reader to \cite{Kato86} or \cite{Rost96}. 

Similarly, we can consider for any $U\subset X$ the kernel of the residue homomorphism (\cite[Chapter 7]{Fasel08a} and \cite[Lemma 9.2.3]{Fasel08a})
$$\xymatrix{I^n(k(U))\ar[r]^-{d_I} & \bigoplus_{x\in U^{(1)}} I^{n-1}(k(x),\omega_{x})}.$$
We denote by $I^n$ the sheaf associated to this presheaf and by $\overline I^n$ the quotient $I^n/I^{n+1}$. Since $d_K$ and $d_I$ are compatible (\cite[Theorem 10.2.6]{Fasel08a}), we also get a residue homomorphism on $G^n(k(U))$ and a sheaf $G^n$ on $X$. Observe that by definition there is an exact sequence of sheaves on $X$
$$\xymatrix{0\ar[r] & I^{n+1}\ar[r] & G^n\ar[r] & K_n\ar[r] & 0.}$$
For the properties of the sheaves $I^{n+1}$ and $G^n$, we refer the reader to \cite[\S 9.3, \S 10.4]{Fasel08a} and \cite{Fasel07}.

The homomorphisms $s_n:K_n(F)\to \overline I^n(F)$ induce a homomorphism of sheaves $K_n/2K_n\to \overline I^n$, which is an isomorphism (again by \cite{Voevodsky03b} and \cite{Orlov07}). These two shaves are yet isomorphic to a third one: For any $U\subset X$ and any prime number $l$, consider the kernel of the residue homomorphism
$$\xymatrix{H_{et}^n(k(U),\mu_l^{\otimes n})\ar[r] & \bigoplus_{x\in U^{(1)}} H^{n-1}_{et}(k(x),\mu_l^{\otimes n-1})}$$
defined for instance in \cite{Rost96}. We denote by $\mathcal H^n(l)$ the sheaf associated to this presheaf. There is a natural homomorphism $K_n\to \mathcal H^n(l)$ (\cite[Lemma 6.1]{Milnor69}), which yields a homomorphism $K_n/l\to \mathcal H^n(l)$.  This is an isomorphism by Voevodsky's work \cite{Voevodsky03b} and \cite{Orlov07} when $l=2$, by Merkurjev and Suslin when $n=2$ (\cite{Merkurjev82}), and by Voevodsky-Suslin with Weibel's patch when $l\neq 2$ and $n\geq 3$ (\cite{Weibel09}). One of the features of the sheaf $\mathcal H^n(l)$ is the Bloch-Ogus spectral sequence (\cite{Bloch74}). This spectral sequence is defined at page $2$ by $E_2^{p,q}:=H^p_{Zar}(X,\mathcal H^q(l))$ and converges to $H_{et}^{p+q}(X,\mu_l)$.

\subsection{Comparison theorems}

Let $A$ be a smooth algebra of dimension $d$ over a field $k$. Recall from \cite[Theorem 2.1]{Fasel08b} that $Um_{d+1}(A)/E_{d+1}(A)$ can be identified with morphisms of schemes $\spec A\to \A^{d+1}-\{0\}$ up to naive homotopies. A straightforward computation shows that $H^d(\A^{d+1}-\{0\},G^{d+1})=GW(k)$, where the latter is the Grothendieck-Witt group of $k$ (\cite[\S 3.3]{Fasel08b}). We obtain in this way a map 
$$\phi:Um_{d+1}(A)/E_{d+1}(A)\to H^d(A,G^{d+1})$$
defined by $v\mapsto v^*(\langle 1\rangle)$, where $v:\spec A\to \A^{d+1}-\{0\}$ is a unimodular row and $v^*:H^d(\A^{d+1}-\{0\},G^{d+1})\to H^d(A,G^{d+1})$ is the pull-back homomorphism defined in \cite[Definition 7.1]{Fasel07}. It turns out that $\phi$ is a weak Mennicke symbol and therefore induces a homomorphism (\cite[Theorem 4.1]{Fasel08b})
$$\Phi:WMS_{d+1}(A)\to H^d(A,G^{d+1}).$$

\begin{thm}\label{wms}
Let $A$ be a smooth algebra of dimension $d\geq 3$ over a perfect field $k$. Then the homomorphism
$$\Phi:WMS_{d+1}(A)\to H^d(A,G^{d+1})$$
is an isomorphism.
\end{thm}

\begin{proof}
See \cite[Theorem 4.9]{Fasel08b}.
\end{proof}

\begin{rem}
In \cite{Fasel08b}, we could only prove that $WMS_{d+1}(A)\to H^d(A,G^{d+1})$ was an isomorphism for $d\geq 3$. However, this result is also true for $d=2$ as we will show in further work. 
\end{rem}

We will now prove that the universal Mennicke symbol $MS_{d+1}(A)$ is isomorphic to $H^d(A,K_{d+1})$ under the same hypotheses as in Theorem \ref{wms}. We first define a map 
$$\psi:Um_{d+1}(A)/E_{d+1}(A)\to H^d(A,K_{d+1})$$
as the composition
$$\xymatrix{Um_{d+1}(A)/E_{d+1}(A)\ar[r]^-\phi & H^d(A,G^{d+1})\ar[r] & H^d(A,K_{d+1})}$$
where the right homomorphism is the homomorphism induced by the map of sheaves $G^{d+1}\to K_{d+1}$.

\begin{lem}
The map $\psi$ induces a homomorphism
$$\Psi:MS_{d+1}(A)\to H^d(A,K_{d+1}).$$
\end{lem}

\begin{proof}
It suffices to prove that $\psi$ satisfies relation ms2. We can follow \cite[proof of Theorem 4.1]{Fasel08b} with $K_{d+1}$ instead of $G^{d+1}$ to get the result.
\end{proof}

\begin{thm}\label{comparison}
Let $A$ be a smooth algebra of dimension $d\geq 3$ over a perfect field $k$ of characteristic different from $2$. Then $\Psi:MS_{d+1}(R)\to H^d(A,K_{d+1})$ is an isomorphism.
\end{thm}

\begin{proof}
Observe first that the following diagram is commutative
$$\xymatrix{WMS_{d+1}(A)\ar[r]^-\Phi\ar[d]_-f & H^d(A,G^{d+1})\ar[d] \\
MS_{d+1}(A)\ar[r]_-\Psi & H^d(A,K_{d+1}).}$$
Indeed, $H^d(A,K_{d+1})$ being a Mennicke symbol, it is also a weak Mennicke symbol.  There is then a unique homomorphism $WMS_{d+1}(A)\to H^d(A,K_{d+1})$ extending $\psi$. Since both compositions in the diagram do the job, they are equal.

Consider now the following commutative diagram:
$$\xymatrix{0\ar[r] & C\ar[r]\ar@{-->}[d] & WMS_{d+1}(A)\ar[r]^-f\ar[d]_-\Phi & MS_{d+1}(A)\ar[r]\ar[d]^-\Psi & 0\\
0\ar[r] & C^\prime\ar[r] & H^d(A,G^{d+1})\ar[r] & H^d(A,K_{d+1})\ar[r] & 0}$$
where $C$ and $C^\prime$ are the kernels of the horizontal homomorphisms. Since $\Phi$ is an isomorphism by Theorem \ref{wms}, it suffices to prove that the induced homomorphism $C\to C^\prime$ is surjective to conclude. Using the exact sequence of sheaves 
$$\xymatrix{0\ar[r] & I^{d+2}\ar[r] & G^{d+1}\ar[r] & K^M_{d+1}\ar[r] & 0,}$$
we get a surjective homomorphism $H^d(A,I^{d+2})\to C^\prime$.

Let $\mathfrak m$ be a maximal ideal of $A$. Arguing as in \cite[Corollary 2.4]{Bhatwadekar02}, we see that there exists a regular sequence $(v_1,\ldots,v_d)\subset A$ such that 
$$A/(v_1,\ldots,v_d)=A/\mathfrak m\times A/M_1\times \ldots \times A/M_r$$
where the $M_i$ are $\mathfrak m_i$-primary ideals for some distinct maximal ideals $\mathfrak m_1,\ldots, \mathfrak m_r$ (also distinct of $\mathfrak m$). Let 
$$\psi_{v_1,\ldots,v_d}:A/(v_1,\ldots,v_d)\to \exten dA{A/(v_1,\ldots,v_d)}A$$
be the isomorphism defined by $\psi_{v_1,\ldots,v_d}(1)=Kos(v_1,\ldots,v_d)$, where the latter is the Koszul complex associated to the regular sequence $(v_1,\ldots,v_d)$. Using the same argument as in \cite[\S 4.1]{Fasel08b}, we see that $I_{fl}^2(A/\mathfrak m)$ is generated by elements of the form $\langle -1,a\rangle\otimes \langle -1,b\rangle \cdot \psi_{v_1,\ldots,v_d}$ with $(a,v_1,\ldots,v_d)$ and $(b,v_1,\ldots,v_d)$ unimodular. Doing the same for all maximal ideals $\mathfrak m\subset A$, we get a set of generators of $H^d(A,I^{d+2})$. The image in $C^\prime$ of such a generator is of the form
$$(ab,\langle -1,ab\rangle\cdot \psi_{v_1,\ldots,v_d})-(a,\langle -1,a\rangle\cdot \psi_{v_1,\ldots,v_d})-(b,\langle -1,b\rangle\cdot \psi_{v_1,\ldots,v_d}).$$
This is precisely the image of 
$$wms(ab,v_1,\ldots,v_d)-wms(a,v_1,\ldots,v_d)-wms(b,v_1,\ldots,v_d)$$ 
under the homomorphism $\Phi:WMS_{d+1}(A)\to H^d(A,G^{d+1})$. Now by definition the element $wms(ab,v_1,\ldots,v_d)-wms(a,v_1,\ldots,v_d)-wms(b,v_1,\ldots,v_d)$ vanishes in $MS_{d+1}(A)$, whence the result.
\end{proof}

This cohomological description allows to perform some computations of the universal Mennicke symbol of length $d+1$, as we will see in the next section.


\section{Prestabilization and stably free modules}\label{stab}

\subsection{Preliminary computations}
As in the previous section, $A$ is a smooth algebra of dimension $d$ over a perfect field $k$ with $\carac k\neq 2$. 

\begin{thm}\label{cd2}
Let $k$ be a perfect field such that $c.d._2(k)\leq 2$ and let $A$ be a smooth affine algebra dimension $d\geq 3$ over $k$. Then the homomorphism 
$$WMS_{d+1}(A)\to MS_{d+1}(A)$$ 
is an isomorphism. 
\end{thm}

\begin{proof}
In view of Theorem \ref{comparison}, it suffices to show that the homomorphism 
$$H^d(A,G^{d+1})\to H^d(A,K_{d+1})$$ 
is an isomorphism. Using the exact sequence of sheaves
$$\xymatrix{0\ar[r] & I^{d+2}\ar[r] & G^{d+1}\ar[r] & K^M_{d+1}\ar[r] & 0,}$$
we see that it suffices to prove that $H^d(A,I^{d+2})=0$. Observe that the group $H^d(A,I^{d+2})$ can be computed by using a flasque resolution of $I^{d+2}$, the (filtered) Gersten-Witt complex of $A$ (\cite[Theorem 3.11]{Fasel07})
$$\xymatrix@C=1.0em{I^{d+2}(k(X))\ar[r] & \ldots\ar[r] & \bigoplus_{x_{d-1}\in X^{(d-1)}}I^3(k(x_{d-1}), \omega_{x_{d-1}})\ar[r] & \bigoplus_{x_d\in X^{(d)}}I^2(k(x_d), \omega_{x_d})}$$
where $X=\spec A$.

If $x_i\in X^{(i)}$, then $k(x_i)$ is of cohomological dimension $c.d._2(k(x_i))\leq d+2-i$ by \cite[\S4.2, Proposition 11]{Serre94}. It follows that $H_{et}^{d+3-i}(k(x_i),\mu_2)=0$ and then $I^{d+3-i}(k(x_i))/I^{d+4-i}(k(x_i))=0$ by \cite[Theorem 4.1]{Orlov07} and \cite[Theorem 7.4]{Voevodsky03b}. The Arason-Pfister Hauptsatz (\cite{Arason71}) then shows that $I^{d+3-i}(k(x_i))=0$. Therefore $H^i(A,I^{d+2})=H^i(A,\overline I^{d+2})$ for any $i\in\N$ and $H^i(A,\overline I^j)=0$ for any $i\in\N$ and $j\geq d+3$. Now $H^i(A,\overline I^j)\simeq H^i(A,\mathcal H^j(2))$ (as seen in Section \ref{cohomology}), and inspection of the Bloch-Ogus spectral sequence shows that $H^d(A,\overline I^{d+2})\simeq H_{et}^{2d+2}(A,\mu_2)$.  The latter is trivial since $A$ is affine of dimension $d$ over a field of cohomological dimension at most $2$ (\cite[Chapter VI, Theorem 7.2]{Milne80} and \cite[Chapter III, Theorem 2.20]{Milne80}). 
\end{proof}

Following the arguments of \cite[Proposition 6.1]{Fasel10b}, we can also prove the following theorem.

\begin{thm}\label{unique}
Let $A$ be a smooth algebra of dimension $d\geq 3$ over a perfect field $k$ with $c.d.(k)\leq 1$. Then $MS_{d+1}(A)$ is uniquely divisible prime to the characteristic of $k$.
\end{thm}

\begin{proof}
Let $l$ be a prime number with $l\neq \mathrm{char}(k)$. We consider the exact sequences of sheaves (see \cite[proof of Corollary 1.11]{Bloch81} for instance) given by the multiplication by $l$ on sheaves
$$\xymatrix{0\ar[r] & {}_{l}K_{d+1}\ar[r] & K_{d+1}\ar[r] & lK_{d+1}\ar[r] & 0}$$
and
$$\xymatrix{0\ar[r] & lK_{d+1}\ar[r] & K_{d+1}\ar[r] & K_{d+1}/l\ar[r] & 0.}$$
The second sequence yields a long exact sequence in cohomology which ends with
$$\xymatrix@C=1.2em{H^{d-1}(A,K_{d+1}/l)\ar[r] & H^d(A,lK_{d+1})\ar[r] & H^d(A,K_{d+1})\ar[r] & H^d(A,K_{d+1}/l)\ar[r] & 0. }$$
Now the groups $H^i(A,K_{d+1}/l)$ are isomorphic to $H^i(A,\mathcal H^{d+1}(l))$(see Section \ref{cohomology}). As in the above proof, $k(x_i)$ is of cohomological dimension $c.d._2(k(x_i))\leq d+1-i$ by \cite[\S4.2, Proposition 11]{Serre94} when $x_i\in X^{(i)}$. Inspecting then the Bloch-Ogus spectral sequence, we get an epimorphism $H^{2d}_{et}(A,\mu_l)\to H^{d-1}(A,K_{d+1}/l)$ and an isomorphism $H^d(A,K_{d+1}/l)\simeq H^{2d+1}_{et}(A,\mu_l)$. Since $A$ is affine, both \'etale cohomology groups are trivial by \cite[Chapter VI, Theorem 7.2]{Milne80} and \cite[Chapter III, Theorem 2.20]{Milne80}. It follows then that $H^d(A,lK_{d+1})\simeq H^d(A,K_{d+1})$.

We use next the first exact sequence of sheaves. This yields 
$$\xymatrix@C=1.2em{H^{d-1}(A,lK_{d+1})\ar[r] & H^d(A,{}_{l}K_{d+1})\ar[r] & H^d(A,K_{d+1})\ar[r] & H^d(A,lK_{d+1})\ar[r] & 0.}$$
Using our result in the previous paragraph, this sequence becomes
$$\xymatrix@C=1.2em{H^{d-1}(A,lK_{d+1})\ar[r] & H^d(A,{}_{l}K_{d+1})\ar[r] & H^d(A,K_{d+1})\ar[r]^-{\cdot l} & H^d(A,K_{d+1})\ar[r] & 0.}$$
It suffices then to prove that the left homomorphism is surjective to conclude. The second exact sequence of sheaves yields a homomorphism $H^{d-2}(A,K_{d+1}/l)\to H^{d-1}(A,lK_{d+1})$ and we consider the diagram
$$\xymatrix{ H^{d-2}(A,K_{d+1}/l)\ar[d]\ar@{-->}[rd]^-f & \\
H^{d-1}(A,lK_{d+1})\ar[r] & H^d(A,{}_{l}K_{d+1}). }$$
Suppose that $k$ contains a primitive $l$-th root of unity $\xi$. For any field $k\subset F$ and any $n\geq 1$, there is a homomorphism $K_{n-1}(F)/l\to {}_{l}K_n(F)$ defined on symbols by $\{a_1,\ldots,a_{n-1}\}\mapsto \{\xi,a_1,\ldots,a_{n-1}\}$. It is not hard to see that this induces a morphism of sheaves $K_{d}/l\to {}_{l}K_{d+1}$. It follows from \cite[Theorem 1.8]{Suslin87} that this morphism of sheaves induces an isomorphism $H^d(A,K_d/l)\to H^d(A,{}_{l}K_{d+1})$. If we denote by $E_2^{pq}$ the terms at page 2 in the Bloch-Ogus spectral sequence,  the homomorphism $f$ then reads as
$$f:E_2^{d-2,d+1}\to E_2^{d,d}$$
and is precisely the differential at page $2$ by \cite[Proposition 7.5]{Barbieri96}. Now $E_3^{d,d}$ injects in $H^{2d}_{et}(A,\mu_l)$, and we have seen above that the right hand term is trivial. Therefore $f$ is surjective and
$$\cdot l:H^d(A,K_{d+1})\to  H^d(A,K_{d+1})$$
is an isomorphism. 

If $k$ doesn't contain a primitive $l$-th root of unity, it is enough to show that $f:H^{d-2}(A,K_{d+1}/l)\to H^d(A,{}_{l}K_{d+1})$ is still surjective to conclude. There exists a finite separable extension $k\subset L$ of degree prime to $l$ such that $L$ contains a primitive $l$-th root of unity. If we denote by $A_L$ the algebra $A\otimes_k L$, we see that the morphism $g:\spec {A_L}\to \spec A$ is finite and \'etale. A simple computation shows that the composition $g_*g^*:H^d(A,{}_{l}K_{d+1})\to H^d(A,{}_{l}K_{d+1})$ is the multiplication by $[L:k]$ and is therefore surjective. It follows that $g_*$ is surjective and we can conclude from the commutative diagram
$$\xymatrix{H^{d-2}(A_L,K_{d+1}/l)\ar[r]^-f\ar[d]_-{g_*} & H^d(A_L,{}_{l}K_{d+1})\ar[d]^-{g_*}    \\
H^{d-2}(A,K_{d+1}/l)\ar[r]_-f & H^d(A,{}_{l}K_{d+1})}$$
that the lower $f$ is surjective.  
\end{proof}

When the base field is of cohomological dimension greater than $2$, then Theorem \ref{cd2} is no longer true. As a simple illustration of this fact, we consider varieties over $\R$. Recall that a smooth connected variety $X$ of dimension $d$ is rational if $X\times \spec\C$ is birational to $\P^d_{\C}$. We begin with the computation of $MS_{d+1}$.

\begin{thm}\label{real_odd}
Let $A$ be a smooth rational $\R$-algebra of dimension $d$. Then 
$$MS_{d+1}(A)\simeq H^d(X(\R),\Z/2)$$
where $X=\spec A$ and the right hand term denotes the singular cohomology of the real manifold $X(\R)$.
\end{thm}

\begin{proof}
First recall from \cite[Proposition 5.4]{Fasel08b} that there is an exact sequence 
$$\xymatrix@C=1.5em{H^d(X\times \spec \C,K_{d+1})\ar[r]^-{f_*} & H^d(X,K_{d+1})\ar[r] & H^d(X,K_{d+1}/2K_{d+1})\ar[r] & 0}$$
where $f_*$ is the push-forward associated to the (finite) morphism $X\times \spec\C\to X$. Since $A$ is rational, $H^d(X\times \spec \C,K_{d+1})=0$ by \cite[Proposition 5.6]{Fasel08b}. Now $H^d(X,K_{d+1}/2K_{d+1})\simeq H^d(X,\mathcal H^{d+1}(2))$ by Section \ref{cohomology} and we get an isomorphism $H^d(X,K_{d+1})\simeq H^d(X,\mathcal H^{d+1}(2))$. Finally, $H^d(X,\mathcal H^{d+1}(2))\simeq H^d(X(\R),\Z/2)$ by \cite[(19.5.1)]{Scheiderer94}.
\end{proof}

In \cite[Theorem 5.7]{Fasel08b}, the group $WMS_{d+1}(A)$ is computed when $A$ has a trivial canonical bundle. We have $WMS_{d+1}(A)\simeq H^d(X(\R),\Z)$. It follows that the homomorphism $WMS_{d+1}(A)\to MS_{d+1}(A)$ is not injective.

\subsection{Prestabilization}

Let $A$ be a smooth algebra of dimension $d$ over a field $k$. The natural homomorphism $SL_{d+1}(A)\to SK_1(A)$ is then surjective, and its kernel is $SL_{d+1}(A)\cap E_{d+2}(A)$ by \cite{Vaserstein69}. A matrix in $SL_{d+1}(A)\cap E_{d+2}(A)$ is called \emph{$1$-stably elementary of size $d+1$}. Our goal in this section is to understand this kernel in a cohomological way when the dimension $d$ of $A$ is odd.

Recall from \cite[Theorem 5.3 (ii)]{vdKallen89} that the map $r:SL_{d+1}(A)\to WMS_{d+1}(A)$ sending a matrix to its first row is a well defined homomorphism. Moreover, when $d$ is odd, we can interpret \cite[Theorem 6.1(ii)]{vdKallen89} by saying that there is an exact sequence
$$\xymatrix{SL_{d+1}(A)\cap E_{d+2}(A)\ar[r]^-r & WMS_{d+1}(A)\ar[r] & MS_{d+1}(A)\ar[r] & 0.}$$
It follows that the kernel $K$ of the map $WMS_{d+1}(A)\to MS_{d+1}(A)$ computes exactly the $1$-stably elementary matrices of size $d+1$ whose first row is not completable in an elementary matrix. When $A$ is of odd dimension $d$ over $\R$ and oriented (i.e. $\wedge^d\Omega_{A/\R}\simeq A$), there are a lot of these matrices.

\begin{thm}
Let $A$ be a smooth affine algebra of odd dimension over $\R$ and let $X=\spec A$. Suppose that $A$ is rational and oriented. Let $\mathcal C$ be the set of compact components of $X(\R)$. Then $K\simeq \oplus_{C\in \mathcal C}\Z$.
\end{thm}

\begin{proof}
First observe that $H^d(A,I^{j})\simeq \oplus_{C\in \mathcal C}\Z$ for any $j\geq d$ (\cite[Proposition 5.1]{Fasel08b}). Moreover, the natural homomorphism $H^d(A,I^{j+1})\to H^d(A,I^{j})$ is multiplication by $2$ on any compact component of $A(\R)$ (and hence is injective). Now the homomorphism $H^d(A,G^j)\to H^d(A,I^j)$ is an isomorphism for $j\geq d$ (\cite[Theorem 5.7]{Fasel08b}). We can now conclude that the sequence 
$$\xymatrix{H^d(A,I^{d+2})\ar[r] & H^d(A,G^{d+1})\ar[r] & H^d(A,K_{d+1})\ar[r] & 0}$$
is also exact on the left. Hence $K\simeq H^d(A,I^{d+2})\simeq \oplus_{C\in \mathcal C}\Z$.
\end{proof}

\begin{exem}
If $A=\R[x_1,\ldots,x_{d+1}]/(\sum x_i^2-1)$ is an algebraic real sphere of dimension $d$, then \cite[proof of Corollary 5.12]{Fasel08b} shows that $WMS_{d+1}(A)$ is generated by $[x_1,\ldots,x_{d+1}]$. If $d=3$, then $(x_1,\ldots,x_4)$ is the first row of the symplectic matrix (see \cite[Theorem 6.11]{Fasel10})
$$\begin{pmatrix} x_1 & x_2 & x_3 & x_4 \\ -x_2 & x_1 & -x_4 & x_3 \\ -x_3 & x_4 & x_1 & -x_2 \\ -x_4 & -x_3 & x_2 & x_1\end{pmatrix}.$$
The first row of the square of this matrix
$$\begin{pmatrix}2x_1^2-1 & 2x_1x_2 & 2x_1x_3 & 2x_1x_4 \\ -2x_1x_2 & 2x_1^2-1 & -2x_1x_4 & 2x_1x_3 \\ -2x_1x_3 & 2x_1x_4 & 2x_1^2-1 & -2x_1x_2 \\ -2x_1x_4 & -2x_1x_3 & 2x_1x_2 & 2x_1^2-1\end{pmatrix}$$
represents $2[x_1,\ldots,x_4]$ in $WMS_4(A)$. Hence its first row is not the first row of an elementary matrix. However
$$\begin{pmatrix}2x_1^2-1 & 2x_1x_2 & 2x_1x_3 & 2x_1x_4 & 0 \\ -2x_1x_2 & 2x_1^2-1 & -2x_1x_4 & 2x_1x_3 & 0 \\ -2x_1x_3 & 2x_1x_4 & 2x_1^2-1 & -2x_1x_2 & 0 \\ -2x_1x_4 & -2x_1x_3 & 2x_1x_2 & 2x_1^2-1 & 0\\ 0 & 0 & 0 & 0 & 1\end{pmatrix}$$
is elementary. 
\end{exem}

In contrast, we have:

\begin{thm}
Let $k$ be a perfect field such that $c.d._2(k)\leq 2$ and let $A$ be a smooth affine algebra of dimension $d\geq 3$ over $k$. Then the first row of any $1$-stably elementary matrix of size $d+1$ is completable in an elementary matrix.
\end{thm}

\begin{proof}
If $d$ is odd, then it is a straightforward consequence of Theorem \ref{cd2}. If $d$ is even, this is \cite[\S 6, Theorem 10]{Vaserstein87}.
\end{proof}

\begin{rem}
If $c.d._2(k)\leq 1$ and $(d+1)!\in k^\times$, then a stronger result is proved in \cite{Rao94}.
\end{rem}

\subsection{About stably free modules}

Recall from Section \ref{umn} that the pointed set $Um_n(R)/SL_n(R)$ computes the set of isomorphism classes of projective $R$-modules $P$ such that $P\oplus R\simeq R^n$. A consequence of Theorem \ref{comparison} is that we get a cohomological description of $Um_{d+1}(A)/SL_{d+1}(A)\cap E_{d+2}(A)$ when $A$ is a smooth algebra of odd dimension $d$.

\begin{prop}
Let $A$ be a smooth algebra of odd dimension $d$ over a perfect field $k$ with $\carac k\neq 2$. Then $Um_{d+1}(A)/SL_{d+1}(A)\cap E_{d+2}(A)\simeq H^d(A,K_{d+1})$. 
\end{prop}

\begin{proof}
This is an easy consequence of \cite[Theorem 6.1(ii)]{vdKallen89}.
\end{proof}

In view of Theorem \ref{real_odd}, we get the following corollary.

\begin{cor}
Let $A$ be a smooth algebra of odd dimension over $\R$. Suppose moreover that $A$ is rational and let $X=\spec A$. Then 
$$Um_{d+1}(A)/SL_{d+1}(A)\cap E_{d+2}(A)\simeq H^d(X(\R),\Z/2).$$  

\end{cor}

Using the homomorphism $r:SL_{d+1}(A)\to WMS_{d+1}(A)$ sending a matrix to its first row, we get a description of $Um_{d+1}(A)/SL_{d+1}(A)$ over any perfect field. 

\begin{thm}
Let $A$ be a smooth affine algebra of odd dimension $d$ over a perfect field $k$. Then $Um_{d+1}(A)/SL_{d+1}(A)\simeq H^d(A,K_{d+1})/SL_{d+1}(A)$.
\end{thm}

\begin{rem}
This result is wrong when the dimension $d$ is even. For example, when $A$ is the real algebraic sphere of dimension $2$, then $Um_3(A)/SL_3(A)=\Z$ by \cite[Corollary 5.12]{Fasel08b}. On the other hand, $H^2(X,K_3)=\Z/2$. 
\end{rem}

\begin{rem}
Under the assumptions of the theorem, we have the stabilization isomorphism $SL_{d+1}(A)/SL_{d+1}(A)\cap E_{d+2}(A)\simeq SK_1(A)$. The first row homomorphism then yields a homomorphism 
$$r:SK_1(A)\to H^d(A,K_{d+1})$$
whose cokernel is precisely $Um_{d+1}(A)/SL_{d+1}(A)$.
\end{rem}


\section{The Bass-Kubota theorem for surfaces}\label{BKS}

\subsection{Grothendieck-Witt groups}

In this section, we briefly recall some basic facts about Grothendieck-Witt groups. The reader is referred to \cite[\S 8]{Schlichting10} for more informations. If $X$ is a scheme over $k$ and $L$ is a line bundle over $X$, then there are abelian groups $GW_i^j(X,L)$ for any $i\in\Z$ and $j\in \Z/4$ generalizing the classical Grothendieck-Witt group $GW(X)$. If $X=\spec R$, then $GW_i^0(R)=K_iO(R)$ while $GW_i^2(R)=K_iSp(R)$. The other groups correspond with Karoubi's $U$ and $V$ groups (\cite[\S 3.6]{Schlichting08pre}). To compare Quillen $K$-theory and higher Grothendieck-Witt groups, there are two functors
$$f:GW_i^j(X,L)\to K_i(X)$$
and 
$$H:K_i(X)\to GW_i^j(X,L)$$
respectively called \emph{forgetful} functor and \emph{hyperbolic} functor. In this setting, Karoubi periodicity reads as a long exact sequence
$$\xymatrix@C=1.2em{\ldots\ar[r] & GW_i^j(X,L)\ar[r]^-f & K_i(X)\ar[r]^-H & GW_i^{j+1}(X,L)\ar[r]^-\eta & GW_{i-1}^j(X,L)\ar[r]^-f & K_{i-1}(X)\ar[r] & \ldots}$$
A basic tool to study Grothendieck-Witt groups is the Gersten-Grothendieck-Witt spectral sequence (\cite[Theorem 25]{Fasel09c} or \cite[Theorem 1.7]{Hornbostel08}). Namely, for any regular scheme $X$ and any $n\in \Z$ there exists a spectral sequence $E(n)^{pq}$ converging to $GW^n_{n-p-q}(X)$ whose terms at page $1$ are 
$$E(n)^{pq}_1:=\displaystyle{\bigoplus_{x_p\in X^{(p)}} GW^{n-p}_{n-p-q}(k(x_p),\omega_{x_p}).}$$
By construction, the hyperbolic and forgetful functors induce morphism of spectral sequences $H:E^{pq}\to E(n)^{pq}$ and $f:E(n)^{pq}\to E^{pq}$, where $E^{pq}$ is the Brown-Gersten-Quillen spectral sequence in $K$-theory (\cite[Theorem 5.20]{Srinivas96}).

For our purpose we will need only the following two results, the first one being an easy lemma.

\begin{lem}\label{pullback}
Let $F$ be a field. For $0\leq n\leq 2$, the image of the homomorphism $\eta:GW_n^n(F)\to GW_{n-1}^{n-1}(F)$ is equal to $I^n(F)$ and the following diagram  
$$\xymatrix{GW_n^n(F)\ar[r]^-\eta\ar[d]_-f & I^n(F)\ar[d]\\
K_n(F)\ar[r]_-{s_n} & \overline I^n(F).}$$
where $f:GW_n^n(F)\to K_n(F)$ is the forgetful homomorphism is a fibre product.
\end{lem}

\begin{proof}
For $n=0$, this is obvious. For $GW_1^1$, see \cite[Corollary 4.5.1.5]{Barge08} and see \cite[Corollary 6.4]{Suslin87} for $GW^2_2$. 
\end{proof}

\begin{prop}\label{forget}
Let $S$ be a smooth affine surface over an algebraically closed field. Then the forgetful homomorphism $f:K_1Sp(S)\to K_1(S)$ is split injective.
\end{prop}

\begin{proof}
We may assume that $S$ is connected. Recall first  that $K_1Sp(S)=GW_1^2(S)$. We use the Gersten-Grothendieck-Witt spectral sequence $E(2)$ to compute this group. Since $S$ is a surface, the information is concentrated on the lines $q=-1,0,1$. The line $q=1$ is trivial by \cite[Lemma 2.2]{Fasel09c}, while the line $q=-1$ is as follows:
$$\xymatrix{GW_3^2(k(x_0))\ar[r] & \displaystyle{\bigoplus_{x_1\in X^{(1)}} GW_2^1(k(x_1))}\ar[r] &  \displaystyle{\bigoplus_{x_2\in X^{(2)}} GW_1^0(k(x_2))}            }$$
The forgetful functor $f$ induces a commutative diagram 
$$\xymatrix{ GW_3^2(k(x_0))\ar[r]\ar[d]_-f & \displaystyle{\bigoplus_{x_1\in X^{(1)}} GW_2^1(k(x_1))}\ar[r] \ar[d]_-f&  \displaystyle{\bigoplus_{x_2\in X^{(2)}} GW_1^0(k(x_2))}\ar[d]_-f\\
K_3(k(x_0))\ar[r] & \displaystyle{\bigoplus_{x_1\in X^{(1)}} K_2(k(x_1))}\ar[r] &  \displaystyle{\bigoplus_{x_2\in X^{(2)}} K_1(k(x_2))} }$$
By \cite[Lemma 2.4]{Fasel09c}, the forgetful functor $f:GW_i^{i-1}(F)\to K_i(F)$ factors through the $2$-torsion and induces a surjective map $f:GW_i^{i-1}(F)\to \{2\}K_i(F)$ for any $i=1,2$ and any field $F$. Moreover, if $F$ is algebraically closed, then the map $f:GW_1^0(F)\to \{2\}K_1(F)$ is an isomorphism. It follows that the cokernel of 
$$\xymatrix{\displaystyle{\bigoplus_{x_1\in X^{(1)}} GW_2^1(k(x_1))}\ar[r] &  \displaystyle{\bigoplus_{x_2\in X^{(2)}} GW_1^0(k(x_2))}            }$$
is isomorphic to $CH^2(S)/2$, which is trivial by \cite[Lemma 1.2]{Colliot96}. This proves that the spectral sequence induces an isomorphism $K_1Sp(S)\simeq E(2)_2^{1,0}$. Explicitly, $K_1Sp(S)$ is isomorphic to the homology of the complex
$$\xymatrix{GW_2^2(k(x_0))\ar[r] & \displaystyle{\bigoplus_{x_1\in X^{(1)}} GW_1^1(k(x_1))}\ar[r] &  \displaystyle{\bigoplus_{x_2\in X^{(2)}} GW(k(x_2))} .           }$$
Arguing as in the proof of Theorem \ref{cd2}, we see that $I^{d+1-j}(k(x_j))=0$ for any $x\in X^{(j)}$. Lemma \ref{pullback} then shows that the forgetful functor induces an isomorphism of complexes
$$\xymatrix{ GW_2^2(k(x_0))\ar[r]\ar[d]_-f & \displaystyle{\bigoplus_{x_1\in X^{(1)}} GW_1^1(k(x_1))}\ar[r] \ar[d]_-f&  \displaystyle{\bigoplus_{x_2\in X^{(2)}} GW_0^0(k(x_2))}\ar[d]_-f\\
K_2(k(x_0))\ar[r] & \displaystyle{\bigoplus_{x_1\in X^{(1)}} K_1(k(x_1))}\ar[r] &  \displaystyle{\bigoplus_{x_2\in X^{(2)}} K_0(k(x_2))}. }$$
Therefore the forgetful functor $f:K_1Sp(S)\to K_1(S)$ induces an isomorphism $K_1Sp(S)\simeq H^1(S,K_2)$. Using now the Brown-Gersten-Quillen spectral sequence $E^{p,q}$ converging to the $K$-theory of $S$, we see that $E_{\infty}^{1,-2}=H^1(S,K_2)$ and that  $f$ induces in fact an isomorphism $K_1Sp(S)\to E_{\infty}^{1,-2}$. It follows that $f$ is split injective.
\end{proof}

We now have all the tools in hand to prove the analogue of the Bass-Kubota theorem for surfaces over an algebraically closed field. This is the object of the next section.
   
\subsection{Mennicke symbols of length $2$}

The definition of the Mennicke symbols of length $2$ has to take into account the fact that $E_2(R)$ is not normal in $SL_2(R)$.

A Mennicke symbol of length $2$ is a map $ms:Um_2(R)\to G$, where $G$ is a group, such that:
\begin{enumerate}[ms1.]
\item $ms(1,0)=1$ and $ms(v)=ms(wM)$ if $M\in SL_2(R)\cap E(R)$. 
\item $ms(a,c)\cdot ms(b,c)=ms(ab,c)$ for any unimodular rows $(a,c)$ and $(b,c)$.
\end{enumerate}
We write $MS_2(R)$ for the universal Mennicke symbol.  It follows from \cite[\S 4]{Suslin76} that the map
$$\Phi:Um_2(R)\to SK_1(R)$$
defined by $\Phi(a,b)=\begin{pmatrix} a & b \\ f & e\end{pmatrix}$ where $e,f$ are any elements in $R$ such that $ae-bf=1$ is a Mennicke symbol of length $2$.

A symplectic Mennicke symbol is a map $msp:Um_2(R)\to G$, where $G$ is a group, such that:
\begin{enumerate}[msp 1.]
\item $msp(1,0)=1$ and $msp(v)=msp(wM)$ if $M\in SL_2(R)\cap ESp(R)$. 
\item $msp(a,b)\cdot msp(a,c^2)=msp(a,bc^2)$ for any $(a,b),(a,c^2)\in Um_2(R)$.
\end{enumerate}
We denote by $MSp_2(R)$ the universal symplectic Mennicke symbol. By \cite[\S4]{Suslin76} again, the map $\Phi$ above factorizes through $K_1Sp(R)$, yielding a map
$$\Psi:Um_2(R)\to K_1Sp(R)$$
which is a symplectic Mennicke symbol. If $R$ is of dimension $2$, then $Um_4(R)=e_1E_4(R)=e_1ESp_4(R)$ by \cite[\S4]{Suslin76}. It follows that $\Psi$ induces an isomorphism $MSp_2(R)\simeq K_1Sp(R)$. 

\begin{lem}\label{inter}
Let $S$ be a smooth affine surface over an algebraically closed field of characteristic different from $2,3$. Then 
$$SL_2(S)\cap E(S)=SL_2(S)\cap E_3(S)=SL_2(S)\cap ESp_4(S)=SL_2(S)\cap ESp(S).$$
\end{lem}

\begin{proof}
The diagram
$$\xymatrix{0\ar[r] & SL_2(S)\cap ESp(S)\ar[r]\ar@{-->}[d] & SL_2(S)\ar[r]\ar@{=}[d] & K_1Sp(S)\ar[d]^-f \\
0\ar[r] & SL_2(S)\cap E(S)\ar[r] & SL_2(S)\ar[r] & K_1(S)}$$
and Proposition \ref{forget} show that $SL_2(S)\cap ESp(S)=SL_2(S)\cap E(S)$. Now 
$$SL_2(S)\cap ESp(S)=SL_2(S)\cap ESp_4(S)$$ 
by \cite[Theorem 2]{Basu10} and $SL_2(S)\cap E(S)=SL_2(S)\cap E_3(S)$ by \cite[\S3]{Rao94}.
\end{proof}

We can finally prove the analogue of Bass-Kubota theorem for surfaces:

\begin{thm}
Let $S$ be a smooth affine surface over an algebraically closed field of characteristic different from $2,3$. Then 
$$MS_2(S)=Um_2(S)/SL_2(S)\cap E_3(S)=Um_2(S)/SL_2(S)\cap ESp_4(S).$$
\end{thm}

\begin{proof}
Since $SL_2(S)\cap E(S)=SL_2(S)\cap ESp(S)$, we see that the rules ms1. and msp1. are equivalent. It follows that any Mennicke symbol is a symplectic Mennicke symbol. We therefore get a homomorphism
$$\theta:MSp_2(S)\to MS_2(S)$$
such that $\theta (msp(a,b))=ms(a,b)$. This map is surjective by definition. Consider the commutative diagram
$$\xymatrix{MSp_2(S)\ar[d]_-\Psi\ar[r]^-\theta & MS_2(S)\ar[d]^-\Phi \\
K_1Sp(S)\ar[r] & SK_1(S). }$$
Since $\Psi$ and $K_1Sp(S)\to SK_1(S)$ are monomorphisms (Proposition \ref{forget}), it follows that $\theta$ is also injective. We can now use Corollary \ref{inter} to conclude.
\end{proof}

\begin{rem}
Contrary to the Bass-Kubota theorem for smooth affine curves, we do not have $MS_2(S)=SK_1(S)$ in general. Indeed, the Brown-Gersten-Quillen spectral sequence shows that $SK_1(S)=MS_2(S)\oplus MS_3(S)$. The right hand term is uniquely divisible prime to the characteristic of the base field and not trivial in general. 
\end{rem}


\bibliography{../../Bib_general/General}{}
\bibliographystyle{plain}


\end{document}